\documentclass[11pt]{article}
\usepackage{amsmath, amsthm, amscd, amsfonts, amssymb}
\date{ }

\newcommand{\ga}{\Gamma}

\newtheorem{theorem}{Theorem}[section]
\newtheorem{lemma}[theorem]{Lemma}
\newtheorem{corollary}[theorem]{Corollary}

\newtheorem{definition}[theorem]{Definition}

\title{Classification of finite $C\theta\theta$-groups with even order and its application}
\smallskip
\normalsize
\author{{\bf Ali Mahmoudifar}\\
Department of Mathematics, Tehran-North Branch, Islamic Azad University, Tehran, Iran\\ e-mail: alimahmoudifar@gmail.com}
\begin{document}
\maketitle
\begin{abstract}
A finite group of order divisible by $3$ in which
centralizers of $3$-elements are $3$-subgroups will be called a $C\theta\theta$-group. The prime graph (or {\it Gruenberg-Kegel graph}) of a finite group $G$ is denoted by $\ga(G)$ (or ${\rm GK}(G)$) and its a familiar. Also the degrees sequence of $\ga(G)$ is called the degree pattern of $G$ and is denoted by $D(G)$. In this paper, first we classify the finite $C\theta\theta$-groups with even order. Then we show that there are infinitely many $C\theta\theta$-groups with the same degree pattern. Finally, we proved that the simple group ${\rm PSL}(2,q)$ and the almost simple group ${\rm PGL}(2,q)$, where $q>9$ is a power of $3$, are determined by their degree pattern.
%In particular, we conclude that if $G$ is a finite group such that $\ga(G)$ is isomorphic to $\ga(L)$ where $L\in\{{\rm PSL}(2,q), {\rm PGL}(2,q)\}$, then $G$ is isomorphic to $L$.
\end{abstract}
\section{Introduction}
If $n$ is a natural number, then
we denote by $\pi(n)$, the set of all prime divisors of $n$.
Throughout this paper by $G$ we mean a finite group. The set $\pi(|G|)$ is denoted by
$\pi(G)$. Also the set of element orders of $G$ is denoted by
$\pi_e(G)$. We denote by $\mu(S)$, the maximal numbers of $\pi_e(G)$ under the divisibility
relation. The \textit{prime graph} (or {\it Gruenberg-Kegel graph}) of $G$ is a simple graph whose vertex
set is $\pi(G)$ and two distinct primes $p$ and $q$ are joined by an
edge (and we write $p\sim q$), whenever $G$ contains an element of
order $pq$. The prime graph of $G$ is denoted by $\ga(G)$.
If $S$ is a finite non-abelian simple group and $G$ is a finite group such that $S\leq G\leq {\rm Aut}(S)$, then we say that $G$ is an almost simple group related to the simple group $S$.The degree pattern of a finite group is defined as follows:
%%%%%%%%%%%%%%%%%%%%%%%%%%%%%%%%%%%%%%%%%%%%%%%%%%%%%%%%%%%%%%%%%%%%%%%%%%%%%%%%%%%%%%%%%%%%%%%%%%%%%%%%%%%%%%%
%\begin{definition}A finite group of order divisible by $3$ in which centralizers of $3$-elements are $3$-groups will be called a $C\theta\theta$-group. \end{definition}
%%%%%%%%%%%%%%%%%%%%%%%%%%%%%%%%%%%%%%%%%%%%%%%%%%%%%%%%%%%%%%%%%%%%%%%%%%%%%%%%%%%%%%%%%%%%%%%%%%%%%%%%%%%%%%%
%\begin{definition} A finite group with an $S_3$-subgroup $M$ containing the centralizer of each of its nonidentity elements will be called a $3CC$-group.\end{definition}
%%%%%%%%%%%%%%%%%%%%%%%%%%%%%%%%%%%%%%%%%%%%%%%%%%%%%%%%%%%%%%%%%%%%%%%%%%%%%%%%%%%%%%%%%%%%%%%%%%%%%%%%%%%%%%%
\begin{definition}{\rm(\cite{mogh})}
Let $G$ be a group and $\pi(G)=\{p_1,p_2,\ldots,p_k\}$ where $p_1<p_2<\cdots<p_k$.
Then the degree pattern of $G$ is defined as follows:
$$D(G):=(\deg(p_1),\deg(p_2),\ldots,\deg(p_k)),$$
where $\deg(p_i)$, $1\leq i\leq k$, is the degree of vertex $p_i$ in $\ga(G)$.
\end{definition}
Following Higman \cite{hig}, a finite group of order divisible by $3$ in which
centralizers of $3$-elements are $3$-groups will be called a $C\theta\theta$-group. A finite
group with a Sylow $3$-subgroup $M$ containing the centralizer of each of its nonidentity
elements will be called a $3CC$-group. So by the above notations, a finite group $G$ is a $C\theta\theta$-group whenever ${\rm deg}(3)=0$ in $\ga(G)$.

We say that a finite group $H$ is determined by its degree pattern of its prime graph if and only if for each finite group $G$ such that $D(G)=D(H)$, then $G$ is isomorphic to $H$. It is clear that the relation $D(G)=D(H)$ does not imply that $\pi(G)=\pi(H)$ (for example we have $D(S_3)=D(D_{10})=(0,0)$ but $\pi(S_3)=\{2,3\}$ and $\pi(D_{10})=\{2,5\}$).
In this paper, first we classify the finite $C\theta\theta$-groups with even order. Then we show that there are infinitely many $C\theta\theta$-groups with the same degree pattern. Also we proved that the simple group ${\rm PSL}(2,q)$ and the almost simple group ${\rm PGL}(2,q)$, where $q>9$ is a power of $3$, ${\rm PSL}(3,4)$ and the simple group ${\rm PSL}(2,p)$ where either $(p-1)/2$ or $(p+1)/2$ is a power of $3$, are determined by their degree pattern. Finally, we ask a question about this type determination of finite groups.
%%%%%%%%%%%%%%%%%%%%%%%%%%%%%%%%%%%%%%%%%%%%%%%%%%%%%%%%%%%%%%%%%%%%%%%%%%%%%%%%%%%%%%%%%%%%%%%%%%%%%%%%%%%%%%%
%%%%%%%%%%%%%%%%%%%%%%%%%%%%%%%%%%%%%%%%%%%%%%%%%%%%%%%%%%%%%%%%%%%%%%%%%%%%%%%%%%%%%%%%%%%%%%%%%%%%%%%%%%%%%%%
\section{Preliminary Results}
%%%%%%%%%%%%%%%%%%%%%%%%%%%%%%%%%%%%%%%%%%%%%%%%%%%%%%%%%%%%%%%%%%%%%%%%%%%%%%%%%%%%%%%%%%%%%%%%%%%%%%%%%%%%%%%
%%%%%%%%%%%%%%%%%%%%%%%%%%%%%%%%%%%%%%%%%%%%%%%%%%%%%%%%%%%%%%%%%%%%%%%%%%%%%%%%%%%%%%%%%%%%%%%%%%%%%%%%%%%%%%%
\begin{lemma}\label{lem-main}\cite[Theram A]{arad-even}
Let $G$ be a $3CC$-group with an $S_3$-subgroup $M$. Then one of
the following statements is true:

(1) $M \unlhd G$;

(2) $G$ has a normal nilpotent subgroup $N$ such that $G/N \cong N_G(M)$,
$M$ cyclic;

(3) $G$ has a normal elementary abelian $2$-subgroup $N$ such that
$G/N\cong {\rm PSL}(2,2^{\alpha})$, $\alpha\geq 2$;

(4) $G\cong {\rm PSL}(2, q)$ for some $q > 3$;

(5) $G \cong {\rm PSL}(3, 4)$;

(6) $G$ contains a simple normal subgroup ${\rm PSL}(2, 3^n)$, $n > 1$, of index $2$.
\end{lemma}
%%%%%%%%%%%%%%%%%%%%%%%%%%%%%%%%%%%%%%%%%%%%%%%%%%%%%%%%%%%%%%%%%%%%%%%%%%%%%%%%%%%%%%%%%%%%%%%%%%%%%%%%%%%%%%%
\begin{lemma}\label{fro}\cite{gor, pas}
Let $G$ be a Frobenius group of even order and let $H$ and $K$ be the
Frobenius complement and Frobenius kernel of $G$, respectively.
Then the following assertions hold:

(a) $K$ is a nilpotent group;

(b) $|K|\equiv 1 \pmod{|H|}$;

(c) Every subgroup of $H$ of order $pq$, with $p$, $q$ (not necessarily distinct) primes, is cyclic. In particular,
every Sylow Subgroup of $H$ of odd order is cyclic and a $2$-Sylow subgroup of $H$ is
either cyclic or a generalized quaternion group. If $H$ is a non-solvable group, then $H$ has a
subgroup of index at most $2$ isomorphic to $Z\times {\rm SL}(2, 5)$, where $Z$ has cyclic Sylow $p$-subgroups
and $\pi(Z) \cap \{2, 3, 5\} = \emptyset$. If $H$ is solvable and $O(H) = 1$, then
either $H$ is a $2$-group or $H$ has a subgroup of index at most $2$ isomorphic to $SL(2, 3)$.
\end{lemma}
%%%%%%%%%%%%%%%%%%%%%%%%%%%%%%%%%%%%%%%%%%%%%%%%%%%%%%%%%%%%%%%%%%%%%%%%%%%%%%%%%%%%%%%%%%%%%%%%%%%%%%%%%%%%%%%
\begin{lemma}\label{2fro}\cite{2fro}
Let $G$ be a 2-Frobenius group, i.e., $G$ has a normal
series $1\lhd H\lhd K\lhd G$, such that $K$ and $G/H$ are Frobenius groups with kernels $H$ and $K/H$,
respectively. If $G$ has even order, then

(i) $G/K$ and $K/H$ are cyclic, $|G/K| \big|  |\textrm{Aut}(K/H)|$
and $(|G/K|,|K/H|)=1$;

(ii) $H$ is a nilpotent group and $G$ is a solvable group.
\end{lemma}
%%%%%%%%%%%%%%%%%%%%%%%%%%%%%%%%%%%%%%%%%%%%%%%%%%%%%%%%%%%%%%%%%%%%%%%%%%%%%%%%%%%%%%%%%%%%
\begin{lemma}\label{mu-pgl}\cite{mu-pgl, pgl}
If $p$ is an odd prime number, then $\mu(PGL(2, p^k)) =
\{p, p^k - 1, p^k + 1\}$ and $\mu(PSL(2, p^k)) =
\{p, (p^k - 1)/2, (p^k + 1)/2\}$. Also we have $\mu(PGL(2, 2^k)) =\mu(PSL(2, 2^k))=
\{2, 2^k - 1, 2^k + 1\}$.
\end{lemma}
%%%%%%%%%%%%%%%%%%%%%%%%%%%%%%%%%%%%%%%%%%%%%%%%%%%%%%%%%%%%%%%%%%%%%%%%%%%%%%%%%%%%%%%%%%%%
\begin{lemma}\label{dio1}\cite{diophan}
The equation $p^m -q^n = 1$, where $p$ and $q$ are primes and $m, n > 1$ has only one solution, namely $3^2 - 2^3 = 1$.
\end{lemma}
%%%%%%%%%%%%%%%%%%%%%%%%%%%%%%%%%%%%%%%%%%%%%%%%%%%%%%%%%%%%%%%%%%%%%%%%%%%%%%%%%%%%%%%%%%%%
\begin{lemma}\label{dio2}\cite{diophan}
With the exceptions of the relations $(239)^2-2(13)^4 = -1$
and $3^5 - 2(11)^2 = 1$ every solution of the equation
\begin{equation*}
p^m - 2q^n = \pm1;~~ p,~ q ~prime; ~~m, n > 1
\end{equation*}
has exponents $m = n = 2$; i.e. it comes from a unit $p-q \cdot 2^{1/2}$ of the quadratic field
$Q(2^{1/2})$ for which the coefficients $p$ and $q$ are primes.
\end{lemma}
%%%%%%%%%%%%%%%%%%%%%%%%%%%%%%%%%%%%%%%%%%%%%%%%%%%%%%%%%%%%%%%%%%%%%%%%%%%%%%%%%%%%%%%%%%%%
%%%%%%%%%%%%%%%%%%%%%%%%%%%%%%%%%%%%%%%%%%%%%%%%%%%%%%%%%%%%%%%%%%%%%%%%%%%%%%%%%%%%%%%%%%%%
\section{Main Results}
%%%%%%%%%%%%%%%%%%%%%%%%%%%%%%%%%%%%%%%%%%%%%%%%%%%%%%%%%%%%%%%%%%%%%%%%%%%%%%%%%%%%%%%%%%%%
%%%%%%%%%%%%%%%%%%%%%%%%%%%%%%%%%%%%%%%%%%%%%%%%%%%%%%%%%%%%%%%%%%%%%%%%%%%%%%%%%%%%%%%%%%%%
\begin{lemma}\label{out-pgl}
Let $G$ be an almost simple group related to the simple group $PSL(2,p^n)$ where $p$ is a prime number and $p^n>3$. If ${\rm deg}(p)=0$ in the prime graph of $G$, then either $G\cong {\rm PSL}(2,p^n)$ or $G\cong {\rm PGL}(2,p^n)$.
\end{lemma}
\begin{proof}
By \cite{kleidman}, we know that ${\rm Out}({\rm PSL}(2 , p^n))\cong \langle \delta\rangle\times \langle \sigma \rangle$, where $\delta$ is a diagonal automorphism with order $2$ and $\sigma$ is field automorphism of ${\rm PSL}(2,p^n)$ with order $n$. Also we know that these automorphisms acts on ${\rm SL}(2,p^n)$ as follows:
\begin{equation*}
 \left(\begin{array}{cc}
a & b  \\
c & d
\end{array} \right)^{\delta}=
\left(\begin{array}{cc}
a & \alpha b  \\
\alpha c & d
\end{array} \right)~~~\&~~~
\left(\begin{array}{cc}
a & b  \\
c & d
\end{array} \right)^{\sigma}=
\left(\begin{array}{cc}
a^p & b^p  \\
c^p & d^p
\end{array} \right),
\end{equation*}
where $\alpha$ is the involution in the multiplicative group of the field $GF(p^n)$, i.e. $GF(p^n)^\#$.
So if in the above relations, we put $b=c=0$ and $a=d$ such that the order of $a$ in $GF(p^n)^\#$ is $p$, then we deduce that any field automorphism $\sigma^f$ and any diagonal-field automorphism $\delta \sigma^f$ such that $1\leq f\leq n$ has a fixed point in ${\rm PSL}(2,p^n)$ with order $p$.

The above discussion implies that when $G$ is an almost simple group related to ${\rm PSL}(2,p^n)$ in which ${\rm deg}(p)=0$ in its prime graph, then $G$ is an extension of this simple group by a diagonal automorphism of ${\rm PSL}(2,p^n)$. Therefore, $G$ is isomorphic to ${\rm PGL}(2,p^n)$.
Finally, by Lemma~\ref{out-pgl}, we get the result.
\end{proof}
%%%%%%%%%%%%%%%%%%%%%%%%%%%%%%%%%%%%%%%%%%%%%%%%%%%%%%%%%%%%%%%%%%%%%%%%%%%%%%%%%%%%%%%%%%%%
\begin{lemma}\label{fro-2fro}
Let $G=K:C$ be a Frobenius group with kernel $K$ and complement $C$. If $C$ is solvable, then the prime graph of $C$ is a complete graph.
\end{lemma}
\begin{proof}
First, let $p$ and $q$ be two different prime numbers included in $\pi(C)$. Also let $p$ and $q$ are not adjacent in $\ga(C)$. Since $C$ is solvable, $C$ has a Hall $\{p,q\}$-subgroup $H$. Without lose of generality, we assume that $O_{p}(H)\neq 1$. So if $H_q$ is a Sylow $q$-subgroup of $H$, then $O_{p}(H):H_q$ is a Frobenius subgroup with kernel $O_{p}(H)$. This implies that $K:(O_{p}(H):H_q)$ is a $2$-Frobenius subgroup of $G$. So by Lemma~\ref{2fro}, we get that $H_q$ has a fixed point in $K$ which is a contradiction since $H_q$ is a subgroup of complement $C$.
\end{proof}
%%%%%%%%%%%%%%%%%%%%%%%%%%%%%%%%%%%%%%%%%%%%%%%%%%%%%%%%%%%%%%%%%%%%%%%%%%%%%%%%%%%%%%%%%%%%
\begin{lemma}\label{fro-me}
Let $G=K:C$ be a Frobenius group with kernel $K$ and complement $C$. If either $3\mid |K|$ or ${\rm deg}(3)=0$ in $\ga(G)$, then $G$ is solvable and $\ga(G)$ has two connected components $\pi(K)$ and $\pi(C)$ which are complete.
\end{lemma}
\begin{proof}
By Lemma~\ref{fro}, it is sufficient to show that $\ga(C)$ is a complete graph. First, let $3\mid |K|$. Thus by Lemma~\ref{fro}, we get that $C$ is solvable and so by Lemma~\ref{fro-2fro}, $\ga(C)$ is complete.

Now let ${\rm deg}(3)=0$ in $\ga(G)$ and $3\nmid |K|$. If $C$ has even order, then $Z(C)\neq1$ which is impossible since ${\rm deg}(3)=0$ in $\ga(G)$. Hence $C$ has odd order and so $C$ is solvable and by Lemma~\ref{fro-2fro}, we get the result.
\end{proof}
%%%%%%%%%%%%%%%%%%%%%%%%%%%%%%%%%%%%%%%%%%%%%%%%%%%%%%%%%%%%%%%%%%%%%%%%%%%%%%%%%%%%%%%%%%%%
\begin{lemma}\label{2fro-graph}
Let $G$ be a 2-Frobenius group with normal
series $1\lhd H\lhd K\lhd G$, such that $K$ and $G/H$ are Frobenius groups with kernels $H$ and $K/H$,
respectively. Then the connected components of the prime graph of $G$ are complete.
\end{lemma}
\begin{proof}
Since $H$ is nilpotent and $G/K$ and $K/H$ are cyclic groups, we may assume that $K$ is a $p$-subgroup and $G/K$ is a $q$-subgroup of $G$ such that $p$ and $q$ are some prime divisors of $|G|$. Hence by Lemma~\ref{2fro}, we get the result.
\end{proof}
%%%%%%%%%%%%%%%%%%%%%%%%%%%%%%%%%%%%%%%%%%%%%%%%%%%%%%%%%%%%%%%%%%%%%%%%%%%%%%%%%%%%%%%%%%%%
\begin{lemma}\label{3cc}
Let $G$ be a finite group with even order. If $3\mid |G|$ and ${\rm deg}(3)=0$ in the prime graph of $G$, then $G$ is a $3CC$-group.
\end{lemma}
\begin{proof}
Since ${\rm deg}(3)=0$, so by the definition $G$ is a $C\theta\theta$-group. Also since $G$ has an even order, by \cite{arad-even}, it follows that $G$ is $3CC$-group.
\end{proof}
%%%%%%%%%%%%%%%%%%%%%%%%%%%%%%%%%%%%%%%%%%%%%%%%%%%%%%%%%%%%%%%%%%%%%%%%%%%%%%%%%%%%%%%%%%%%
\begin{theorem}\label{main1}
Let $G$ be a finite group with even order. If $3$ is an isolated vertex in the prime graph of $G$ and $D(G)$ is the degree pattern of $G$, then:

a) If $G$ is solvable, then one of the following cases holds: 

a-1) $G=K:C$ is a Frobenius group with kernel $K$ and complement $C$. In this case one of the subgroups $K$ or $C$ is a $3$-subgroup of $G$. 

a-2) $G$ is a $2$-Frobenius group with normal series $1\lhd H\lhd K\lhd G$ such that $N$ is a $3'$-subgroup, $K$ is a cyclic $3$-subgroup and $C\cong Z_2$ is cyclic.

Moreover, in these two cases, if we put $n=|\pi(G)|$, then we have
\begin{equation*}
D(G)=(n-2,0,n-2,\dots,n-2)
\end{equation*}

b) if $G$ is non-solvable, then one of the following cases holds:
 
b-1) has a normal elementary abelian $2$-subgroup $N$ such that $G/N$ is isomorphic to either $PSL(2,4)$ or $PSL(2,8)$ and we have
\begin{equation*}
D(PSL(2,4))=D(PSL(2,8))=(0,0,0)
\end{equation*}

b-2) $G$ is isomorphic the simple group $PSL(3,4)$ and we have
\begin{equation*}
D(PSL(3,4))=(0,0,0,0),
\end{equation*}

b-3) $G$ is isomorphic the simple group $PSL(2,p)$ where $p$ is a prime number and either $(p+1)/2$ or $(p-1)/2$ is a power of $3$. In this case, if we put $n=|\pi(PSL(2,p))|$, then we have
\begin{equation*}
D(PSL(2,p))=(n-3,0,n-3,\cdots,n-3,0),
\end{equation*}

b-4) $G$ is isomorphic the simple group $PSL(2,3^n)$ where $n>1$. In this case, if we put $n_1=|\pi((3^n-1)/2)|$ and $n_2=|\pi((3^n+1)/2)|$, then we have
\begin{equation*}
D(PSL(2,3^n))=(n_1-2 ~or~ n_2-2,0,n_1-2 ~or~ n_2-2,\dots,n_1-2 ~or~ n_2-2),
\end{equation*}

b-5) $G$ is isomorphic to the almost simple group $PGL(2,3^n)$, where $n>1$. In this case, if we put $n_1=|\pi(3^n-1)|$ and $n_2=|\pi(3^n+1)|$, then we have
\begin{equation*}
D(PGL(2,3^n))=(n_1+n_2-2,0,n_1-1~ or ~n_2-1,\dots,n_1-1 ~or~ n_2-1).
\end{equation*}
\end{theorem}
\begin{proof}
First we note that since we assume that $3$ is an isolated vertex in the prime graph of $G$, so $G$ is a $C\theta\theta$-group. Also since $G$ has an even order, then by Lemma~\ref{lem-main}, one of the following assertions holds:

{\bf Case 1:} If $M$ is a Sylow $3$-subgroup of $G$, then $M$ is normal. Hence by Schur-Zassenhaus's Theorem, $M$ has a complement $H$ and so $G=M:H$ in which $H$ is a Hall $3'$-subgroup of $G$. Since $3$ is an isolated vertex in the prime graph of $G$, we get that $H$ acts fixed point freely on $M$. This implies that $G$ is a Frobenius group with kernel $M$ and complement $H$. So by Lemma~\ref{fro-me}, we get (a-1).

{\bf Case 2:} $G$ has a normal nilpotent subgroup $N$ such that $G/N\cong N_G(M)$, where $M$ is a Sylow $3$-subgroup of $G$ and $M$ is cyclic.

{\bf Subcase 2-1:} If $M=N_G(M)$, then $G=N:M$ is a Frobenius group with kernel $N$. Also in this case, since $N$ is nilpotent and $M$ is a $3$-subgroup, then $G$ is solvable and again we get (a-1).

{\bf Subcase 2-2:} If $M$ is a pure subgroup of $N_G(M)$, then $N_G(M)=M:C$ is a Frobenius group with kernel $M$. Since $M$ is cyclic, then $C\cong Z_2$. In this case, $G$ is a $2$-Frobenius group with normal series $1\unlhd N\unlhd N:M\unlhd G$ which implies (a-2).   

{\bf Case 3:} $G$ has a normal elementary abelian 2-subgroup $N$ such that $G/N\cong PSL(2,2^{\alpha})$, where $\alpha\geq2$.
By Lemma~\ref{mu-pgl}, we have $\mu(PSL(2,2^{\alpha}))=\{2,2^{\alpha}-1,2^{\alpha}-1\}$. So either $2^{\alpha}-1$ or $2^{\alpha}+1$ is a power of $3$. Thus by Lemma~\ref{dio1}, the only possible cases are either $\alpha=2$ or $\alpha=3$. This asserts that $G/N$ is isomorphic to either $PSL(2,4)$ or $PSL(2,8)$, which gets (b-1).

{\bf Case 4:} $G\cong PSL(2,q)$ where $q$ is a power of odd prime number $p>3$. In this case, by Lemma~\ref{mu-pgl}, $\mu(PSL(2,q))=\{p,q-1,q+1\}$. Since $3$ is an isolated vertex in the prime graph of ${\rm PSL}(2,q)$, we get that either $(q-1)/2$ or $(q+1)/2$ is a power of $3$. Now by Lemma~\ref{dio2}, it follows that $q=p$. So either $p-1$ or $p+1$ is a power of $3$, Case (b-3).

{\bf Case 5:} $G$ is isomorphic to $PSL(2,3^n):\langle\theta\rangle$, where $\theta$ is an involution. This means that $G$ is an almost simple group related to ${\rm PSL}(2,3^n)$ such that ${\rm deg}(3)=0$ in the prime graph of $G$. So by Lemma~\ref{out-pgl}, we conclude that $G\cong PGL(2,3^n)$, Case (b-5).

{\bf Case 6:} $G$ is isomorphic to $PSL(2,3^n)$ where $n>1$ or $PSL(3,4)$. By Lemma~\ref{mu-pgl} and \cite{atlas}, we get that the vertex $3$ is an isolated vertex in the prime graph of these simple groups, Cases (b-2) and (b-4).

By \cite{atlas} and Lemma~\ref{mu-pgl}, we can easily compute the degree pattern of the above groups, which completes the proof.
\end{proof}
%%%%%%%%%%%%%%%%%%%%%%%%%%%%%%%%%%%%%%%%%%%%%%%%%%%%%%%%%%%%%%%%%%%%%%%%%%%%%%%%%%%%%%%%%%%%
We remark that the above theorem does not show that the finite groups explained in that theorem exist. it is obvious that when $p$ is a prime number such that $p\in\{5,7,17,19,53,163\}$, the simple group $PSL(2,p)$ satisfies the condition of Case (b-3) of Theorem~\ref{main1}. However, it is not clear that the number of these finite simple groups is finite or infinite. In the following, we show that there is infinitely many solvable groups with even order such that ${\rm deg}(3)=0$ in their prime graph.
%%%%%%%%%%%%%%%%%%%%%%%%%%%%%%%%%%%%%%%%%%%%%%%%%%%%%%%%%%%%%%%%%%%%%%%%%%%%%%%%%%%%%%%%%%%%
\begin{lemma}\label{infi-fro1}
There exist infinitely many Frobenius groups of even order whose Fitting subgroups are $3$-subgroups.
\end{lemma}
\begin{proof}
Let $F$ be an arbitrary field of characteristic $3$. We know that the multiplicative group $F^\#$ has a unique involution $\gamma$. Also we know that $F^\#$ acts fixed point freely on the additive group $F$. So the semidirect product of $F$ by the subgroup generated by $\gamma$ is a Frobenius group of even order whose Fitting subgroup is a $3$-subgroup.
\end{proof}
%%%%%%%%%%%%%%%%%%%%%%%%%%%%%%%%%%%%%%%%%%%%%%%%%%%%%%%%%%%%%%%%%%%%%%%%%%%%%%%%%%%%%%%%%%%%
similarly to the previous lemma we get the following result:
\begin{lemma}\label{infi-fro2}
There exist infinitely many Frobenius groups of even order whose complement is a cyclic $3$-subgroup.
\end{lemma}
%%%%%%%%%%%%%%%%%%%%%%%%%%%%%%%%%%%%%%%%%%%%%%%%%%%%%%%%%%%%%%%%%%%%%%%%%%%%%%%%%%%%%%%%%%%%
\begin{lemma}\label{infi-2fro}
There exist infinitely many 2-Frobenius groups $G$ with normal series $1\unlhd N\unlhd H\unlhd G$ such that $H/N$ is a $3$-group and ${\rm deg}(3)=0$ in $\ga(G)$.
\end{lemma}
\begin{proof}
Let $F$ be the Galois Field $GF(2^2)=\{0,1,\alpha,\alpha+1\}$, where $\alpha^2=1$. Similarly to the previous lemma, we can construct a Frobenius group with Fitting subgroup $F$ and cyclic component $F^\#=\{1,\alpha,\alpha+1\}$. If we put $C:=\{0,1\}$, then $C$ is a subgroup of Galois Filed $F$. Now we define an action of $C$ on $F^\#$ as follows if $x\in C$ and $y\in F^\#\setminus\{1\}$, then $y^x:=x+y$. This action shows that the semidirect product group $G:=F^\#:C$ is a Frobenius group with Fitting subgroup $F^\#$.

Now let $V=F\bigoplus F\bigoplus\cdots \bigoplus F$ be a vector space with dimension $n$ over the field $F$. So $G$ acts on the additive group $V$ as follows if $\xi_1+\xi_2+\cdots+\xi_n\in V$ and $yx\in G$, where $y\in F^\#$ and $x\in C$, then $(\xi_1+\xi_2+\cdots+\xi_n)^{yx}:=(\xi_1y)+(\xi_2y+\cdots+(\xi_ny)$. This implies that $V:F^\#$ is a Frobenius group with fitting subgroup $V$. So the finite group $V:G$ is a $2$-Frobenius group in which $3$ is an isolated vertex in its prime graph.
\end{proof}
%%%%%%%%%%%%%%%%%%%%%%%%%%%%%%%%%%%%%%%%%%%%%%%%%%%%%%%%%%%%%%%%%%%%%%%%%%%%%%%%%%%%%%%%%%%%
\begin{corollary}\label{D(G)2}
Let $G$ be finite solvable group with even order such that ${\rm deg}(3)=0$ in the prime graph of $G$. If $H$ is a finite group such that $D(H)= D(G)$, then $H$ is solvable.
\end{corollary}
\begin{proof}
By Theorem~\ref{main1}, it is straightforward.
\end{proof}
%%%%%%%%%%%%%%%%%%%%%%%%%%%%%%%%%%%%%%%%%%%%%%%%%%%%%%%%%%%%%%%%%%%%%%%%%%%%%%%%%%%%%%%%%%%%
The next lemma shows that there are infinitely manu non-solvable groups such that ${\rm deg}(3)=0$ in their prime graph.
%%%%%%%%%%%%%%%%%%%%%%%%%%%%%%%%%%%%%%%%%%%%%%%%%%%%%%%%%%%%%%%%%%%%%%%%%%%%%%%%%%%%%%%%%%%%
\begin{lemma}\label{infi-psl}
Let $L$ be a finite group isomorphic to one of the simple groups $PSL(2,4)$ or $PSL(2,8)$. Then there exist infinitely many groups $G$ such that ${\rm deg}(3)=0$ in $\ga(G)$ and $G/O_2(G)$ is isomorphic to $L$.
\end{lemma}
\begin{proof}
Let $L\cong PSL(2,4)\cong SL(2,4)$. there is a modular representation of $L$ over Galois Field $GF(2^2)$ as follows:
\begin{equation*}
G:=\big\{
\left(\begin{array}{ccc}
a & b & u\\
c & d & v \\
0 & o & 1
\end{array}\right)~~|~~where~~ad-bc=1~~and~~u,v\in GF(2^2)
\big\}
\end{equation*}
So $L$ acts fixed point freely on the vector space $V:=GF(2^2)\bigoplus GF(2^2)$. Hence the finite group $V:L$ is a group in which $3$ is an isolated vertex in its prime graph. Now we similarly, we can construct a finite group $G:=(V\bigoplus V\bigoplus\cdots \bigoplus V):L$ such that $G/O_2(G)\cong L$ and ${\rm deg}(3)=0$ in $\ga(G)$.
With a similar argument we can make such a finite group when $L\cong PSL(2,8)$
\end{proof}
%%%%%%%%%%%%%%%%%%%%%%%%%%%%%%%%%%%%%%%%%%%%%%%%%%%%%%%%%%%%%%%%%%%%%%%%%%%%%%%%%%%%%%%%%%%%
\begin{corollary}\label{D(G)1}
Let $G$ be finite non-solvable group such that ${\rm deg}(3)=0$ in the prime graph of $G$. If $H$ is a finite group such that $D(H)= D(G)$, then one of the following cases holds:

1) $H$ has a normal $2$-subgroup $N$, such that $H/N$ is isomorphic to either ${\rm PSL}(2,4)$ or ${\rm PSL}(2,8)$. Moreover, in this case we have $D(H)=(0,0,0)$.

2) $H$ is isomorphic to $G$. 
\end{corollary}
\begin{proof}
Using the degree pattern of groups in Theorem~\ref{main1}, it is straightforward.
\end{proof}
%%%%%%%%%%%%%%%%%%%%%%%%%%%%%%%%%%%%%%%%%%%%%%%%%%%%%%%%%%%%%%%%%%%%%%%%%%%%%%%%%%%%%%%%%%%%

\begin{corollary}
If $G$ is a finite group such that $D(G)=D(\emph{PGL}(2,3^k))$, where $k>2$ is a natural number, then $G\cong\emph{PGL}(2,3^k)$.
\end{corollary}
%%%%%%%%%%%%%%%%%%%%%%%%%%%%%%%%%%%%%%%%%%%%%%%%%%%%%%%%%%%%%%%%%%%%%%%%%%%%%%%%%%%%%%%%%%%
\begin{corollary}
If $G$ is a finite group such that $\ga(G)=\ga(\emph{PGL}(2,3^k))$, where $k>2$ is a natural number, then $G\cong\emph{PGL}(2,3^k)$.
\end{corollary}
%%%%%%%%%%%%%%%%%%%%%%%%%%%%%%%%%%%%%%%%%%%%%%%%%%%%%%%%%%%%%%%%%%%%%%%%%%%%%%%%%
%%%%%%%%%%%%%%%%%%%%%%%%%%%%%%%%%%%%%%%%%%%%%%%%%%%%%%%%%%%%%%%%%%%%%%%%%%%%%%%%%%%

\end{document}